\documentclass[11pt,reqno]{amsart} 

\usepackage{amsmath,amssymb,amsthm,mathtools}
\usepackage[utf8]{inputenc}
\usepackage{graphicx,caption,subcaption}
\usepackage{xcolor}
\usepackage{multirow}
\usepackage{hyperref}
\hypersetup{
    colorlinks=true,
    linkcolor=blue,
    citecolor=blue,
    urlcolor=blue
}

\theoremstyle{plain}
\newtheorem{theorem}{Theorem}[section]
\newtheorem{lemma}[theorem]{Lemma}

\newtheorem{proposition}[theorem]{Proposition}

\theoremstyle{definition}
\newtheorem{definition}[theorem]{Definition}

\theoremstyle{remark}
\newtheorem{remark}[theorem]{Remark}

\begin{document}

\title{Spectral Diffusion Models on the Sphere}

\author{Francesco Mari}
\address{Department of Statistical Sciences, Sapienza University of Rome, Italy}
\email{francesco.mari@uniroma1.it}

\author{Pierpaolo Brutti}
\address{Department of Statistical Sciences, Sapienza University of Rome, Italy}
\email{pierpaolo.brutti@uniroma1.it}

\author{Claudio Durastanti}
\address{Department of Basic and Applied Sciences for Engineering, Sapienza University of Rome, Italy}
\email{claudio.durastanti@uniroma1.it}

\begin{abstract}
Diffusion models provide a principled framework for generative modeling via stochastic
differential equations and time-reversed dynamics. However, extension of spectral diffusion
approaches to spherical data raises nontrivial geometric and stochastic
issues that are absent in the Euclidean setting.

In this work, we develop a diffusion modeling framework defined directly on
finite-dimensional spherical harmonic representations of real-valued functions on
the sphere. We show that the spherical discrete Fourier transform maps spatial Brownian
motion to a constrained Gaussian process in the frequency domain with deterministic,
generally non-isotropic covariance. This induces modified forward- and reverse-time
stochastic differential equations in the spectral domain.

As a consequence, spatial and spectral score matching objectives are generally no longer equivalent, even in the band-limited setting. We establish a quantitative relationship between the two objectives, showing that the geometry-induced covariance of the spectral noise gives rise to a distinct, geometry-dependent inductive bias. We also derive the corresponding forward and reverse diffusion equations and characterize the induced noise covariance.
\end{abstract}

\maketitle

\section{Introduction}

Deep generative models have become a central approach for learning complex high-dimensional probability distributions. Among them, diffusion models have emerged as one of the most successful frameworks, achieving state-of-the-art performance in a wide range of generative tasks (see, for example, ~\cite{ho20,hyvarinen05,sohl15,song20}). Their formulation in terms of stochastic differential equations (SDEs) and reverse-time stochastic dynamics provides both strong empirical performance and a mathematically transparent probabilistic framework.

Among the various formulations of diffusion models, score-based generative modeling via stochastic differential equations (SDEs) provides a unified and mathematically transparent framework, see~\cite{song20}. In this setting, a forward SDE progressively transforms the data distribution to a tractable reference measure, typically a standard Gaussian, while generation is achieved by simulating the associated reverse-time SDE, whose drift depends on the score function \(\nabla_x \log p_t(x)\). This continuous-time formulation has recently received rigorous treatment in the mathematical literature through the theory of time-reversed stochastic differential equations, further strengthening its probabilistic foundations. 

Crucially, diffusion models are not invariant in the choice of representation. Changing the representation space modifies the driving noise and, consequently, the reverse-time dynamics (see, for example, \cite{deb22}). Different representations therefore induce different geometric and statistical structures on the diffusion process, leading to distinct inductive biases and different generative performance (cf., among others, \cite{rombach2022}). More generally, the influence of geometry on diffusion dynamics has recently attracted increasing attention in the study of stochastic processes on structured and non-Euclidean spaces (cf.~\cite{ding25}).

While diffusion models were originally developed for Euclidean data such as images and audio, there is growing interest in representations adapted to the underlying data geometry. Time series provide a relevant example, where the frequency domain naturally captures long-range dependencies and multiscale structure. Recent work has shown that diffusion models defined directly in the frequency domain can preserve the score-matching framework while improving the modeling of global temporal structure (see, among others~\cite{crabbe24}), suggesting that spectral diffusion introduces a distinct inductive bias rather than merely providing an alternative parameterization.

Many important datasets are naturally defined on spherical domains. Examples include full-sky cosmological observations such as cosmic microwave background and large-scale structure maps~\cite{cm24,PDKS19}, global climate and Earth observation fields~\cite{cocs17,uaskaog24}, molecular surface representations in computational chemistry~\cite{mgm08,wbags11}, and omnidirectional visual data captured by 360\(^\circ\) cameras~\cite{cgc18,mbamr22}. For such data, spherical harmonics provide the canonical spectral representation, naturally encoding the geometry and rotational symmetries of the sphere. Nevertheless, extending spectral diffusion from Euclidean or temporal settings to spherical data is substantially more challenging. Unlike the Euclidean Fourier transform, the spherical harmonic transform of real-valued signals imposes conjugate symmetry constraints and, under discrete quadrature, maps spatial Brownian motion to a correlated Gaussian process in the spectral domain. Consequently, the driving noise is no longer isotropic, and both the forward and reverse-time SDEs, as well as the associated score matching objective, must be properly modified. 

While harmonic analysis has long provided a natural framework for the study of stochastic processes on the sphere, including stationary spherical time series and ARMA-type models (see, for example, \cite{cm21,cdv21}), these works do not address the forward and reverse stochastic dynamics underlying score-based generative modeling. Existing work on score-based diffusion over Riemannian manifolds similarly focuses on intrinsic diffusion processes~\cite{deb22}. In contrast, our goal is to formulate intrinsically nonstationary score-based diffusion processes in a finite-dimensional spherical harmonic representation, where the geometry is encoded through spectral constraints and the induced noise covariance.

Despite the success of spectral diffusion models for Euclidean and sequential data, score-based diffusion models defined directly in the spherical harmonic domain are comparatively unexplored.
In particular, a systematic formulation of score-based diffusion directly in the spherical harmonic domain, accounting for symmetry constraints, induced noise covariance, and the resulting forward and reverse-time dynamics, is largely missing from the literature.  
This motivates the development of a score-based diffusion framework intrinsically adapted to spherical harmonic representations, together with a theoretical characterization of the corresponding learning objectives.

In this work, we propose a diffusion modeling framework defined directly in the space of spherical harmonic coefficients. We show that the spherical Fourier transform maps standard Brownian motion in the spatial domain to a constrained stochastic process in the frequency domain, which we characterize as a spherical mirrored Brownian motion with nontrivial covariance structure. Building on this characterization, we derive the corresponding forward and reverse-time SDEs and characterize the induced covariance structure. Our main theoretical result establishes a quantitative relationship between score-matching objectives in the spatial and spectral domains, showing that they are generally not equivalent because of the geometry-induced covariance of the spectral noise. The resulting framework is fully compatible with the score-based SDE formalism, while naturally encoding spherical geometry and symmetry constraints.  
It therefore provides a principled foundation for generative modeling of spherical data, with potential applications in cosmology, climate science, computational chemistry, and computer graphics.

Several natural extensions of the present framework merit future investigation.
While all constructions in this work are finite-dimensional with a fixed band
limit $L$, an important open question concerns the asymptotic behavior of the
spectral diffusion process as $L \to \infty$, and its relationship to diffusions
defined intrinsically on $L^2(\mathbb{S}^2)$. Related directions include the study
of robustness with respect to the choice of spherical sampling and quadrature
schemes, as well as extensions to anisotropic noise, other compact homogeneous
spaces, or alternative spectral representations.

The remainder of the paper is organized as follows.  
Section~\ref{sec:background} reviews score-based diffusion models and the elements of harmonic analysis required to define diffusion processes on \( \mathbb{S}^2 \).  
Section~\ref{sec:contribution} presents our main theoretical results, developing a matrix-based formulation of the spherical discrete Fourier transform, deriving the associated forward and reverse diffusion SDEs in the frequency domain, and establishing the relationship between spatial and spectral score-matching objectives. Section~\ref{numerics} provides illustrative numerical experiments on synthetic, band-limited spherical signals, designed to complement the theoretical analysis and highlight the geometric and stochastic effects induced by the spectral formulation. 
Section~\ref{sec:proofs} contains all proofs.

\section{Background}\label{sec:background} 
To establish the foundations of our spectral diffusion framework on the sphere, we first review score-based generative modeling in Euclidean spaces (see, among others,~\cite{hyvarinen05,song19}), and then review the elements of harmonic analysis required to extend these ideas to spherical domains (cf.~\cite{mp11,ah12}). \\ 
Throughout the paper, vectors in Euclidean space are denoted by bold lowercase letters \( \mathbf{x} \in \mathbb{R}^n \), while complex-valued vectors are written \( \mathbf{z} \in \mathbb{C}^d \), with \( \mathfrak{R}[\mathbf{z}] \) and \( \mathfrak{I}[\mathbf{z}] \) denoting their real and imaginary parts and $\mathbf{z}^*$ the complex conjugate of $\mathbf{z}$.  
Time-dependent stochastic processes are written \( \mathbf{x}(t) \),  indexed by a continuous diffusion variable $t\in[0,T]$. With
the slight abuse of notation we shall abbreviate  their marginal densities \( p_t(\mathbf{x}(t)) \) by $p_t(\mathbf{x})$. Finally, $\mathbf{w}(t)$ denotes a standard $n$-dimensional Brownian motion.

\subsection{Score-Based Generative Modeling in Euclidean Space} 
Score-based diffusion models construct generative mechanisms by coupling stochastic dynamics with learned gradient fields of evolving probability densities.\\  
The central idea is to define a forward-time stochastic process that gradually transforms the data distribution into a simple reference measure, and then to invert this process by simulating the corresponding reverse-time dynamics.

\paragraph{Forward Diffusion.} 
Let \( \mathbf{x} \in \mathbb{R}^n \) denote data drawn from an unknown density \( p_{\mathrm{data}} \).  \\
In continuous-time diffusion models, the forward process is defined by the stochastic differential equation
\begin{equation}\label{eq:forwardSDE}
d\mathbf{x}_t = \mathbf{f}(\mathbf{x}_t,t)dt + \mathbf{G}(t)d\mathbf{w}_t,
\end{equation}
where \( \mathbf{f}(\mathbf{x}_t,t) \) and \( \mathbf{G}(t) \) are a drift term and a time-dependent diffusion matrix respectively.

We denote \( p_t \) the marginal density of the solution \( \mathbf{x}(t) \) of \eqref{eq:forwardSDE}, with initial condition \( p_0 = p_{\mathrm{data}} \).  
The functions \( \mathbf{f} \) and \( \mathbf{G} \) are chosen so that, as \( t \to T \), the distribution \( p_t \) converges to the density \( p_T \) of an isotropic standard Gaussian.

\paragraph{Reverse Diffusion.} 
Generation is achieved by reversing the forward diffusion.  \\
Starting from samples drawn from \( p_T \), one simulates a reverse-time stochastic process whose marginals evolve toward \( p_0 \).  
As shown in the classical work~\cite{anderson82}, the reverse-time dynamics satisfy
\begin{equation}\label{eq:revdif}
d\mathbf{x}_t = \mathbf{b}(\mathbf{x}_t,t) dt + \mathbf{G}(t) d\hat{\mathbf{w}_t},
\end{equation}
where the reverse-time process evolves from \( T \) to \( 0 \), \( \hat{\mathbf{w}} \) is Brownian motion in reverse time, and
\[
\mathbf{b}(\mathbf{x}_t,t)
=
\mathbf{f}(\mathbf{x}_t,t)
-
\mathbf{G}(t)\mathbf{G}(t)^{\mathsf T}
\nabla_{\mathbf{x}} \log p_t(\mathbf{x}).
\]
The reverse drift depends explicitly on the score function \( \nabla_{\mathbf{x}} \log p_t(\mathbf{x}) \).  
Once this score is known or well approximated, the reverse SDE can be simulated to generate samples from the data distribution.

\paragraph{Score Estimation.}
Since the intermediate densities \( p_t \) are unknown, score-based diffusion models learn an approximation of the score function directly from data.  
This is achieved via denoising score matching~\cite{hyvarinen05,song19}, by training a neural network \( s_\theta(\mathbf{x},t) \) to minimize
\begin{equation}\label{eq:score_matching}
\theta^*
\!=
\underset{\theta}{\mathrm{argmin}}
\, \mathbb{E}_t
\! \left[
\lambda(t)\,
 \mathbb{E}_{\mathbf{x}(0)}
\mathbb{E}_{\mathbf{x}(t)\vert \mathbf{x}(0)}\!
\left\|
s_\theta\left( \mathbf{x}(t),t \right)
\!-\!
\nabla_{\mathbf{x}(t)} \!\log p_{0,t}\!\left( \mathbf{x}(t)\big\vert \mathbf{x}(0)\right)
\right\|_2^2
\right]\!,
\end{equation}
where \( p_{0,t} \) denotes the transition kernel of the forward SDE.  
Under mild regularity assumptions and sufficient model capacity, the minimizer recovers the true score almost everywhere.

As aforementioned, this perspective highlights that diffusion models are fundamentally constrained by the geometry and representation of the data space, especially outside the Euclidean case (see, for example, \cite{deb22,rombach2022}).

\subsection{Harmonic Analysis on the Sphere}
Spherical harmonics provide the natural analogue of the Fourier basis for functions defined on the unit sphere \( \mathbb{S}^2 \). 
They form an orthonormal basis of \( L^2(\mathbb{S}^2) \), allowing spherical data to be represented through frequency components indexed by degree and order. \\ 
This spectral decomposition is central to both theoretical analysis and practical modeling, as it separates global and local features, enables the construction of rotationally invariant operators, and supports scale-dependent regularization and stochastic dynamics.

Harmonic representations on the sphere have long played a fundamental role in statistical inference for spherical data, with applications ranging from spectral estimation in Gaussian models~\cite{dlm14} to spherical time series and functional autoregressions~\cite{cm21}.  \\
From the perspective of generative modeling, these properties make spherical harmonics a natural vehicle for transferring diffusion-based methodologies from Euclidean spaces to spherical domains, while respecting intrinsic geometry and rotational symmetries.

\paragraph{Spherical Harmonics.}
Let \( L^2(\mathbb{S}^2) \) denote the space of real-valued square-integrable functions on the unit sphere, equipped with the inner product
\[
\langle f,g\rangle
=
\int_{\mathbb{S}^2} f(\theta,\phi)g^*(\theta,\phi)d\Omega(\theta,\phi),
\qquad
d\Omega = \sin\theta d\theta d\phi,
\]
where \( (\theta,\phi)\in [0,\pi]\times[0,2\pi) \) denote colatitude and longitude respectively.

An orthonormal basis of \( L^2(\mathbb{S}^2) \) is given by the spherical harmonics \[\left\{Y_{\ell,m}(\theta,\phi):\ell \in \mathbb{N}_0,m=-\ell,\ldots,\ell\right\},\] indexed by degree \( \ell \in \mathbb{N}_0 \) and order \( |m|\le\ell \), defined as
\[
Y_{\ell,m}(\theta,\phi)
=
\sqrt{\frac{2\ell+1}{4\pi}\frac{(\ell-m)!}{(\ell+m)!}}P_{\ell,m}(\cos\theta)e^{im\phi},
\qquad m\ge0,
\]
with the extension to \( m<0 \) via conjugate symmetry.  
Here \( P_{\ell,m} \) are the associated Legendre functions~\cite{szego75}.

\noindent Any \( f\in L^2(\mathbb{S}^2) \) admits the spherical harmonic expansion
\[
f(\theta,\phi)
=
\sum_{\ell=0}^\infty
\sum_{m=-\ell}^{\ell}
a_{\ell,m} Y_{\ell,m}(\theta,\phi),
\qquad
a_{\ell,m}
=
\langle f,Y_{\ell,m}\rangle.
\]
The coefficients $\left\{a_{\ell,m}:\ell \in \mathbb{N}_0,m=-\ell,\ldots,\ell\right\}$ encode angular frequency content, with \( \ell \) controlling spatial scale and \( m \) azimuthal oscillations.

The following properties are particularly relevant for spectral modeling:
\begin{itemize}
\item \textit{Orthonormality:} \( \langle Y_{\ell,m},Y_{\ell',m'}\rangle=\delta_{\ell\ell'}\delta_{mm'} \).
\item \textit{Conjugate symmetry:} \( Y_{\ell,-m}=(-1)^m Y^*_{\ell,m} \).
\item \textit{Eigenfunction property:} \( \Delta_{\mathbb{S}^2}Y_{\ell,m}=-\ell(\ell+1)Y_{\ell,m} \), where $\Delta_{\mathbb{S}^2}$ is the spherical Laplacian, that is the Laplace-Beltrami operator on $\mathbb{S}^2$.
\end{itemize}
In view of the symmetry property for the spherical harmonics, also the coefficients satisfy the Hermitian conjugate symmetry
\begin{equation}\label{eq:symm}
\hat{a}_{\ell,m} = (-1)^m \hat{a}^*_{\ell,-m}, \quad 0<\ell<L, \ |m|<\ell,
\end{equation}
ensuring that the reconstructed spatial function $f(\theta,\phi)$ is real-valued.
\paragraph{Band-Limited Functions and Sampling.}
In practice, we work with band-limited functions for which \( a_{\ell,m}=0 \) for \( \ell\ge L \).  
Such functions admit exact reconstruction from finitely many samples using appropriate quadrature schemes~\cite{dh94}.

Throughout this work, we adopt a uniform sampling scheme following~\cite{mcewen12}, which provides explicit sampling nodes and quadrature weights enabling exact analysis and synthesis of band-limited functions. More in detail, for a function band-limited at degree $L$, exact recovery is possible from a finite set of samples. Following \cite{mcewen12}, we introduce uniform sampling nodes in colatitude and longitude,
\[
\theta_j = \frac{(2j+1)\pi}{4L}, \quad j=0,\dots,2L-1, 
\qquad 
\phi_k = \frac{2\pi k}{2L-1}, \quad k=0,\dots,2L-2,
\]
corresponding to evenly spaced points in colatitude and longitude, respectively.  

To account for the spherical integration measure, we associate quadrature weights
\[
q_j^{(L)} = \frac{2}{L} \sin\theta_j 
\sum_{\ell=0}^{L-1} \frac{1}{2\ell+1}\sin\big( (2\ell+1)\theta_j\big),
\]
which ensure exact integration of spherical harmonics of degree $\ell<L$.  

The resulting analysis formula recovers the harmonic coefficients exactly,
\begin{equation}\label{eqn:analysis}
\hat{a}_{\ell,m}
=
\frac{2\pi}{2L-1}
\sum_{j=0}^{2L-1}\sum_{k=0}^{2L-2}
q_j^{(L)} f(\theta_j,\phi_k) Y_{\ell,m}^*(\theta_j,\phi_k),
\qquad \ell<L,
\end{equation}
while the corresponding \emph{synthesis formula} reconstructs the band-limited function on the grid,
\begin{equation}\label{eqn:synthesis}
f(\theta_j,\phi_k)
=
\sum_{\ell=0}^{L-1}\sum_{m=-\ell}^{\ell}
\hat{a}_{\ell,m} Y_{\ell,m}(\theta_j,\phi_k).
\end{equation}
This exact quadrature scheme underlies practical spherical Fourier transforms and discrete spectral methods, and provides the computational basis for probabilistic modeling in the harmonic domain.
\begin{remark}
Throughout this section we adopt the equiangular sampling grid of~\cite{mcewen12}, which admits exact quadrature for $L$-band-limited functions and leads to a particularly simple Kronecker structure for the weight matrix $Q$.  
This choice is made for notational and computational convenience only.  
All results extend to any spherical sampling scheme admitting exact quadrature for $L$-band-limited functions, such as Gauss--Legendre or related schemes~\cite{ln97}, provided the corresponding spherical harmonics matrix $Y$ and weight matrix $Q$ are defined accordingly. In all cases, it is essential to balance the number of sampling points with the bandwidth \( L \) to avoid aliasing effects~\cite{dur24,dp19}.
The diffusion formalism developed below depends only on the linearity of the spherical DFT and on the induced covariance structure of the transformed noise.
\end{remark}

\paragraph{Probability Densities in the Spectral Domain.} 
To extend diffusion models to the spectral domain, we define a probability density over the complex-valued spherical harmonic coefficients $\hat{a}_{\ell,m} \in \mathbb{C}$. These coefficients are collected in $\hat{\mathbf{a}} \in \mathbb{C}^{\hat{d}_X}$, where 
\[
\hat{d}_X = \sum_{\ell=0}^{L-1} (2\ell+1) = L^2,
\]
counts the total number of independent coefficients. Following the standard treatments of complex random vectors~\cite{scs10}, we can write the score in terms of real and imaginary parts:
\[
\hat{\mathbf{s}}(\hat{\mathbf{a}}) := 
\nabla_{\mathfrak{R}[\hat{\mathbf{a}}]} \log \hat{p}(\hat{\mathbf{a}}) 
+ i  \nabla_{\mathfrak{I}[\hat{\mathbf{a}}]} \log \hat{p}(\hat{\mathbf{a}}),
\]
with $\mathfrak{R}[\cdot]$ and $\mathfrak{I}[\cdot]$ denoting the real and imaginary components. This allows treating complex coefficients as structured real vectors, suitable for score-based methods in the spectral domain.

 Not all coefficients are independent: the admissible set forms a linear subspace (equivalently, a lower-dimensional embedded submanifold) of $\mathbb{C}^{\hat{d}_X}$, and the corresponding score function must respect the mirrored symmetry given by Equation \eqref{eq:symm}
\[
\hat{\mathbf{s}}_{\ell,m} = (-1)^m \hat{\mathbf{s}}^*_{\ell,-m},
\]
where $\hat{\mathbf{s}}_{\ell,m}=[\hat{\mathbf{s}}( \hat{\mathbf{a}})]_{\ell,m}$, for the sake of brevity.\\ 
This constraint guarantees that the stochastic dynamics of the spectral diffusion process preserve the reality of the underlying spatial function. By working directly with these constrained coefficients, both forward and reverse diffusions can be defined in the harmonic domain. This naturally encodes the sphere's geometry and symmetries while retaining all information needed for exact reconstruction. Such formulation underpins the spectral diffusion framework developed throughout this work.

These symmetry constraints will play a crucial role in determining the covariance structure of the driving noise in the spectral diffusion SDEs derived in Section~\ref{sec:contribution}.

\section{Diffusion Models in the Frequency Domain for Spherical Data}
\label{sec:contribution}

In the previous section, we reviewed the score-based diffusion framework in Euclidean spaces and recalled the elements of harmonic analysis required to represent functions on the sphere. 
We now specialize to real-valued functions defined on the unit sphere $\mathbb{S}^2$ and show how diffusion in the frequency domain can be formulated in the spherical setting. All constructions in this section are finite-dimensional, with a fixed band limit $L$; asymptotic regimes as $L \to \infty$ are left for future work. \\
To this end, we first derive a matrix representation of the spherical discrete Fourier transform, and then use it to define forward and reverse diffusion processes directly in the spectral domain.

Throughout the remainder of this work, we consider an equiangular sampling grid on the sphere as in~\cite{mcewen12}, consisting of $N_\theta \cdot N_\phi$ points, where
\[
N_\theta = 2L,
\qquad
N_\phi = 2L-1,
\]
denote the numbers of colatitude and longitude samples, respectively.

We work with samples of a real-valued spherical signal arranged in vector form as
\[
\mathbf{x}^{\mathsf T}
=
\left(
f(\theta_0,\phi_0),
f(\theta_0,\phi_1),
\dots,
f(\theta_{N_\theta-1},\phi_{N_\phi-1})
\right),
\qquad
\mathbf{x}\in\mathbb{R}^{N_\theta N_\phi}.
\]
The corresponding vector of reconstructed spherical harmonic coefficients is denoted by
\begin{equation}
\label{eq: spherical_harmonic_coeff}
\hat{\mathbf{a}}^{\mathsf T}
=
\left(
\hat{a}_{0,0},
\hat{a}_{1,0},
\hat{a}_{1,1},
\hat{a}_{1,-1},
\hat{a}_{2,0},
\hat{a}_{2,1},
\hat{a}_{2,-1},
\hat{a}_{2,2},
\dots,
\hat{a}_{L-1,-L+1}
\right),
\qquad
\hat{\mathbf{a}}\in\mathbb{C}^{L^2}.    
\end{equation}
Finally, we denote by
\[
d_X = N_\theta N_\phi = 2L(2L-1)
\qquad\text{and}\qquad
\hat{d}_X = \sum_{\ell=0}^{L-1}(2\ell+1)=L^2,
\]
the dimensions of the spatial and spectral representations, respectively.

\subsection{Matrix expression for the spherical DFT}    
The crucial point is that the spherical discrete Fourier transform (DFT),
$\hat{\mathbf{a}} = \mathcal{F}[\mathbf{x}]$, is \emph{linear} with respect to
the sampled signal $\mathbf{x}$. This linearity allows the transform to be
represented as a matrix operator acting on the vector of spatial samples.

To make this representation explicit, we introduce two matrices associated with
the spherical DFT: a full complex matrix $Y$, collecting the values of the
spherical harmonics evaluated on the sampling grid, and a diagonal matrix $Q$,
encoding the quadrature weights required for exact integration of
$L$-band-limited functions on the sphere (see
Section~\ref{sec:background}).

To keep the notation concise, from now on and until the end of this section we
adopt the following analysis formula:
\begin{equation}
    \hat{a}_{\ell, m}
    \;=\;
    \sum_{j=0}^{2L-1}\sum_{k=0}^{2L-2}
    q_{j}^{(L)}\;
    f(\theta_j,\phi_k)\;
    Y_{\ell m}^*(\theta_j,\phi_k),
\end{equation}
which expresses each spectral coefficient as a weighted inner product between
the sampled field and the corresponding spherical harmonic. The quadrature weights are given by
\[q_j^{(L)}
    = \frac{2\pi}{2L-1}
    \frac{2}{L}
    \sin\theta_j
    \sum_{\ell=0}^{L-1}
    \frac{1}{2\ell+1}\sin \big(  (2\ell+1)\theta_j\big) \quad \text{for} \quad j=0,\dots,2L-1.
    \]
Compared to the standard formulation, the prefactor $\frac{2\pi}{2L-1}$ has been
absorbed into the definition of the quadrature weights, so that the remaining
expressions take a compact matrix form.

\begin{definition}[Spherical harmonics matrix]
\label{def:shmat}
Let $f \in L^2(\mathbb{S}^2)$ be $L$-band-limited, and let
$f(\theta_j,\phi_k)$ denote its samples for
$j = 0,\dots,N_\theta-1$ and $k = 0,\dots,N_\phi-1$.
The spherical harmonics matrix
$Y \in \mathbb{C}^{d_X \times \hat{d}_X}$ is defined component-wise by
\[
    [Y]_{(j,k),(\ell,m)} = Y_{\ell m}(\theta_j,\phi_k).
\]
\end{definition}

The matrix $Y$ collects all spherical harmonics up to degree $L-1$ evaluated at
the sampling locations. Columns correspond to spectral modes $(\ell,m)$, according to the same order of coefficients in \eqref{eq: spherical_harmonic_coeff}, while
rows correspond to spatial grid points $(\theta_j,\phi_k)$. This matrix plays
the role of a spherical analogue of the Fourier matrix in the Euclidean DFT.

\begin{definition}[Sample weights matrix]
\label{def:samplewmat}
Under the assumptions of Definition~\ref{def:shmat}, let
\[
    Q_\theta
    =    \operatorname{diag}(q_0,q_1,\dots,q_{N_\theta-1})
    \in  \mathbb{R}^{N_\theta \times N_\theta}.
\]
The sample weights matrix
$Q \in \mathbb{R}^{d_X \times d_X}$ is defined by
\begin{equation}
    Q = Q_\theta \otimes I_{N_\phi},
\end{equation}
where $\otimes$ denotes the Kronecker product and
$I_{N_\phi}$ is the identity matrix of size $N_\phi = 2L-1$.
\end{definition}

The Kronecker structure reflects the fact that the quadrature weights depend
only on the colatitude index $j$, while sampling along the longitude direction
is uniform.

We are now ready to derive an equivalent matrix--vector representation of the spherical DFT.

\begin{lemma}[DFT of a real-valued spherical map]
\label{lemma:DFT_mat_rep}
Let $f \in L^2(\mathbb{S}^2)$ be $L$-band-limited, and let 
\[
f(\theta_j,\phi_k)_{j=0,\dots,2L-1, k=0,\dots,2L-2}
\] 
be its samples on the equiangular grid. Define the vector of spatial samples 
\[
\mathbf{x}^{\sf T} = \left(f(\theta_0,\phi_0), f(\theta_0,\phi_1), \dots, f(\theta_{2L-1},\phi_{2L-2})\right) \in \mathbb{R}^{2L(2L-1)},
\] 
and let $\hat{\mathbf{a}} \in \mathbb{C}^{L^2}$ be the vector of spherical Fourier coefficients ordered as 
\[
\hat{\mathbf{a}}^{\sf T} = \left(\hat{a}_{0,0}, \hat{a}_{1,0}, \hat{a}_{1,1}, \hat{a}_{1,-1}, \dots, \hat{a}_{L-1,-L+1}\right).
\] 
Then the discrete spherical Fourier transform can be expressed as the matrix--vector product
\begin{equation}\label{eqn:lemma1}
\hat{\mathbf{a}} = U\mathbf{x}, \qquad \text{with} \quad U := Y^{H} Q,
\end{equation}
where $Y$ is the spherical harmonics matrix and $Q$ the sample weights matrix as in Definitions~\ref{def:shmat} and~\ref{def:samplewmat}.
\end{lemma}
\noindent The proof is available in Section \ref{sec:proofs}

Lemma~\ref{lemma:DFT_mat_rep} provides the matrix-vector counterpart of the analysis formula given in Equation~\eqref{eqn:analysis}. In this form, the operator $U$ acts as the spherical Fourier analysis matrix, combining the evaluation of spherical harmonics with the appropriate quadrature weights.

It is equally straightforward to derive the matrix-vector version of the synthesis formula corresponding to Equation~\eqref{eqn:synthesis}. Indeed, for any $L$-band-limited signal $\mathbf{x}$ with spectral coefficients $\hat{\mathbf{a}}$, one has
\begin{equation}
    \mathbf{x} = Y\hat{\mathbf{a}} .
\end{equation}

At first glance, this might suggest that $Y$ is the inverse of $U$. This conclusion is, however, incorrect. The matrix $U$ is not square: in fact,
\[
U \in \mathbb{C}^{\hat{d}_X \times d_X},
\qquad
\hat{d}_X = L^2,
\qquad
d_X = 2L(2L-1),
\]
so that $U$ maps the higher-dimensional spatial representation into the lower-dimensional spectral one.

We now clarify the precise algebraic relationship between the spherical DFT operator $U$ and the spherical harmonics matrix $Y$, and show how exact reconstruction is nevertheless achieved for band-limited signals despite the dimensional mismatch.

\begin{lemma}[Pseudoinverse of the spherical DFT operator $U$]
\label{lemma:pseudoinverseU}
Let $U \in \mathbb{C}^{\hat{d}_X \times d_X}$ be the spherical DFT operator as defined in Lemma~\ref{lemma:DFT_mat_rep}, let $Y \in \mathbb{C}^{d_X \times \hat{d}_X}$ be the spherical harmonics matrix (Definition~\ref{def:shmat}), and let $Q \in \mathbb{R}^{d_X \times d_X}$ be the sample weights matrix (Definition~\ref{def:samplewmat}). In the standard $L$-band-limited setting, $d_X = 2L(2L-1)$ and $\hat{d}_X = L^2$.
Then:
\begin{itemize}
    \item \textit{Right pseudoinverse:} The spherical harmonics matrix $Y$ is a right pseudoinverse of $U$,
    \begin{equation}
        U Y = I_{\hat{d}_X}.
    \end{equation}
    \item \textit{Left pseudoinverse:} Let $P := Y U$. Then for any sampled signal $\mathbf{x} \in \mathbb{R}^{d_X}$, the following are equivalent:
    \begin{enumerate}
        \item $\mathbf{x}$ is $L$-band-limited; that is, there exists $\hat{\mathbf{a}} \in \mathbb{C}^{\hat{d}_X}$ with $\mathbf{x} = Y \hat{\mathbf{a}}$;
        \item $P \mathbf{x} = \mathbf{x}$.
    \end{enumerate}
\end{itemize}
\end{lemma}
\noindent The proof is given in Section \ref{sec:proofs}

\begin{remark} 
The left pseudoinverse lemma provides a clear characterization of the band-limited subspace. The implication from band-limited $\mathbf{x}$ to $P\mathbf{x} = \mathbf{x}$ corresponds to the usual "analysis then synthesis recovers the signal" identity. Conversely, the implication from $P\mathbf{x} = \mathbf{x}$ to $\mathbf{x}$ being band-limited shows that $P$ acts as the orthogonal projector (with respect to the $Q$-weighted inner product $\langle u,v \rangle_Q = u^H Q v$) onto the band-limited subspace $\operatorname{range}(Y)$.
\end{remark}

Lemma~\ref{lemma:pseudoinverseU} shows that the spherical DFT operator $U$ admits a pseudoinverse and that band-limited signals are precisely characterized by the projector $P = Y U$. This immediately implies a fundamental property of $U$: it preserves inner products between the band-limited spatial space and the Fourier coefficient space, which we formalize in the following proposition.

\begin{proposition}[Isometry between the $L$-band-limited spherical signal space and the space of Fourier coefficients]
\label{prop:U_isom}

Let $U \in \mathbb{C}^{\hat{d}_X \times d_X}$ be the spherical DFT operator defined in Lemma~\ref{lemma:DFT_mat_rep}. Under the same assumptions as in Lemma~\ref{lemma:pseudoinverseU}, let $\langle \cdot , \cdot \rangle_Q$ denote the $Q$-weighted inner product, and let
\[
\left( C^{d_X}_{\mathrm{band\text{-}limited}}, \langle \cdot , \cdot \rangle_Q \right)
\]
be the space of $L$-band-limited signals $\mathbf{x}$, i.e., signals for which there exists a vector of spherical Fourier coefficients $\hat{\mathbf{a}} \in \mathbb{C}^{\hat{d}_X}$ such that $\mathbf{x} = Y \hat{\mathbf{a}}$.

\noindent Then the linear operator $U$ is an isometry between $\left( C^{d_X}_{\mathrm{band\text{-}limited}}, \langle \cdot , \cdot \rangle_Q\right)$ and $\left( \mathbb{C}^{\hat{d}_X}, \langle \cdot , \cdot \rangle\right)$.
\end{proposition}
Unlike the Euclidean DFT, the spherical DFT is an isometry only after restriction to the band-limited subspace and with respect to a weighted inner product, a distinction that becomes crucial in the stochastic setting.
The proof is presented in Section \ref{sec:proofs}.

\subsection{Diffusion SDEs}

Having expressed the Fourier transform for spherical maps in matrix form, we now derive the diffusion SDEs for spherical signals in the frequency domain by applying the spherical DFT operator $U$ to the forward spatial-domain SDE, Equation~\eqref{eq:forwardSDE}. Since the forward SDE is driven by a standard Brownian motion $\mathbf{w}$,
it is crucial to characterize the image of $\mathbf{w}$ under the spherical
DFT operator $U$. We begin by formalizing the class of Gaussian processes
that naturally arise in this setting.
\begin{definition}[Spherical mirrored Brownian motion]
\label{def:spherical_mirrored_bm}
A $\mathbb{C}^{L^2}$-valued stochastic process $\mathbf{v}(t)$ is called a
\emph{spherical mirrored Brownian motion} if it satisfies:
\begin{itemize}
\item conjugate symmetry $\mathbf{v}_{(\ell,m)} = (-1)^m \mathbf{v}^*_{(\ell,-m)}$,
\item independent increments,
\item covariance structure $\mathbb{E}[\varphi(\mathbf{v}(t))\varphi(\mathbf{v}(t))^{\mathsf T}] = t\,\Sigma$,
\end{itemize}
where $\Sigma$ is the covariance matrix induced by the spherical DFT and $\varphi$
extracts the independent real degrees of freedom. The increments are independent in time, although the components may be correlated at fixed times.
\end{definition}
Although $\mathbf{v}(t)$ is complex-valued, its effective dimension is $L^2$, corresponding to the independent real degrees of freedom extracted by $\varphi$.

The following lemma shows that the spherical Fourier transform of a
standard real Brownian motion belongs to this class, and provides an
explicit characterization of its components and covariance structure.

\begin{lemma}[Image of Brownian motion under the spherical DFT]
\label{lemma:SphericalDFT_BM}

Let $\mathbf{x} \in \mathbb{R}^{d_X}$ be samples of a spherical signal with band limit $L$, and let $\mathbf{w}$ be a standard Brownian motion on $\mathbb{R}^{d_X}$. Then the transformed process 
\[
\mathbf{v} = U \mathbf{w}
\] 
is a continuous stochastic process satisfying the following properties:

\begin{enumerate}
    \item \textit{Symmetry.} For all multi-indices $(\ell,m)$, the components satisfy
    \[
        \mathbf{v}_{(\ell,m)} = (-1)^m \mathbf{v}_{(\ell,-m)}^*.
    \]
    \item \textit{Real Brownian motions.} For $0 \le \ell < L$, the rescaled components
    \[
        \mathbf{v}_{(\ell,0)}/\sqrt{2 C_{(\ell,0),(\ell,0)}}
    \]
    are standard real Brownian motions.
    \item \textit{Complex Brownian motions.} For $0 \le \ell < L$ and $0 < m < L$, the components can be written as
    \[
        \mathbf{v}_{(\ell,m)} = \sqrt{C_{(\ell,m),(\ell,m)}} \left( \tilde{\mathbf{w}}^1_{(\ell,m)} + i \tilde{\mathbf{w}}^2_{(\ell,m)} \right),
    \]
    where $\tilde{\mathbf{w}}^1_{(\ell,m)}$ and $\tilde{\mathbf{w}}^2_{(\ell,m)}$ are independent standard Brownian motions.
\end{enumerate}

Here, the factor $C_{(\ell,m),(\ell',m')}$ is defined as
\[
    C_{(\ell,m),(\ell',m')} := \frac{2L-1}{2} \, N_{\ell,m} N_{\ell',m'} \sum_{j=0}^{2L-1} q_j^2 P_{\ell,m}(\cos \theta_j) P_{\ell',m'}(\cos \theta_j).
\]
Note that $C_{(\ell,m),(\ell',m')}$ vanishes unless $m=m'$, so the covariance
structure is block-diagonal with respect to the azimuthal index $m$.
\end{lemma}

\noindent Hence, $\mathbf{v}$ is a spherical mirrored Brownian motion in the sense of
Definition~\ref{def:spherical_mirrored_bm}.

In contrast to the Euclidean setting, and analogously to the spectral
time-series case studied in~\cite{crabbe24}, the process $\mathbf{v}$ is no longer
a standard Brownian motion due to the redundancy imposed by conjugate symmetry, that is, the symmetry between the pairs of components $\mathbf{v}_{\ell,m}$ and $\mathbf{v}_{\ell,-m}$. Nevertheless, as shown below, it is possible to recover a standard Brownian motion structure by restricting attention to a subset of non-redundant components.

\begin{remark}[Densities and scores for constrained spherical Fourier coefficients]  
\label{rem:density_score_constrained}  
We now provide a technical derivation of the density and score on the constrained Fourier coefficient manifold, which underlies the reverse-time SDE, that is introduced below. 

As discussed above, the redundancy between certain components of the spherical DFT $\hat{\mathbf{a}} \in \mathbb{C}^{L^2}$ of $\mathbf{x} \in \mathbb{R}^{2L(2L-1)}$ must be taken into account to define a valid probability density $\hat{p}$ for spherical maps in the frequency domain. This redundancy implies that $\hat{p}$ is naturally defined on the submanifold
\[
\mathbb{C}_{\text{constr}}^{L^2} := \left\{\hat{\mathbf{a}} = (\hat{a}_{0,0}, \dots, \hat{a}_{L-1,L-1}) \;\text{s.t.}\; \hat{a}_{\ell,m} = (-1)^m \hat{a}_{\ell,-m}^*\right\} \subset \mathbb{C}^{L^2}.
\]  

Following \cite{crabbe24}, we define a coordinate chart $\varphi : \mathbb{C}_{\text{constr}}^{L^2} \mapsto \mathbb{R}^{L^2}$ by extracting the unconstrained part of $\hat{\mathbf{a}}$. Concretely, we concatenate the relevant real and imaginary parts of the spherical DFT with $m>0$:
\begin{equation}
\label{eq:phi_chart_remark}
\varphi[\hat{\mathbf{a}}] := \bigoplus_{\ell=0}^{L-1} \left( \Re(a_{\ell,0}),\; \left(\Re(a_{\ell,m}), \Im(a_{\ell,m})\right)_{m=1}^{\ell} \right),
\end{equation}
where $\bigoplus$ denotes vector concatenation. The components satisfy
\[
a_{\ell,0} = \varphi_{\ell^2}, \quad \Re(a_{\ell,m}) = \varphi_{\ell^2 + 2m-1}, \quad \Im(a_{\ell,m}) = \varphi_{\ell^2 + 2m}, \; m = 1,\dots,\ell.
\]

By symmetry, $\hat{\mathbf{a}}$ can be uniquely reconstructed from $\varphi[\hat{\mathbf{a}}]$, defining the inverse map $\varphi^{-1} : \mathbb{R}^{L^2} \mapsto \mathbb{C}_{\text{constr}}^{L^2}$ as 
\begin{equation}
\label{eq:phi_inv_remark}
\varphi^{-1}(\mathbf{z})
=
\bigoplus_{\ell=0}^{L-1}
\left(
\mathbf{z}_{\ell^2},
\;
\left(
\mathbf{z}_{\ell^2+2m-1} + i\mathbf{z}_{\ell^2+2m},
(-1)^m\left(\mathbf{z}_{\ell^2+2m-1} - i\mathbf{z}_{\ell^2+2m}\right)
\right)_{m=1}^{\ell}
\right).
\end{equation}
With this chart, we define a density on the real vector space $\hat{p}_\varphi : \mathbb{R}^{L^2} \mapsto \mathbb{R}^+$ and pull it back to the constrained manifold:  
\[
\hat{p} := \hat{p}_\varphi \circ \varphi.  
\]  
This ensures $\hat{p}$ depends only on the independent real and imaginary components of $\hat{\mathbf{a}}$ and respects the mirror symmetry.  

Finally, the score $\hat{\mathbf{s}} : \mathbb{C}_{\text{constr}}^{L^2} \times [0,T] \mapsto \mathbb{C}^{L^2}$ is constructed analogously. First define the real score on $\mathbb{R}^{L^2}$:
\[
\hat{\mathbf{s}}_\varphi(\mathbf{z},t) = \nabla_{\mathbf{z}} \log \hat{p}_{\varphi,t}(\mathbf{z}), \quad \mathbf{z} \in \mathbb{R}^{L^2}, t \in [0,T],
\]  
and pull it back to the constrained manifold:
\begin{equation}
\label{eq:score_freq}
\hat{\mathbf{s}}(\hat{\mathbf{a}},t) := \varphi^{-1}\big[ \hat{\mathbf{s}}_\varphi(\varphi(\hat{\mathbf{a}}),t)\big], \quad \hat{\mathbf{a}} \in \mathbb{C}_{\text{constr}}^{L^2}.
\end{equation}

This defines a vector field of partial derivatives with respect to the independent real and imaginary parts of $\hat{\mathbf{a}}$ and is fully consistent with the symmetry constraints. It thus provides the rigorous foundation for the score term appearing in the reverse-time SDE, Equation~(\ref{eq:ReverseSDE_Frequency}).  
\end{remark}

Once a probabilistic model and a well-defined score function for complex-valued frequency coefficients have been introduced, we are in a position to formulate diffusion SDEs directly in the spherical Fourier domain. In particular, we leverage Lemma~\ref{lemma:SphericalDFT_BM} to characterize the dynamics of the spherical Fourier coefficients $\hat{\mathbf{a}}$ as SDEs driven by spherical mirrored Brownian motions.

\begin{theorem}[Diffusion process in the frequency domain] 
\label{thm:DiffusionFrequency}
Let $\mathbf{x}$ be a diffusion process solving Equation~\eqref{eq:forwardSDE} and $\mathbf{G}(t) = g(t) I_{\hat{d}_X}$. Then its spherical Fourier transform $\hat{\mathbf{a}} = \mathcal{F}[\mathbf{x}]$ satisfies the forward diffusion SDE in the frequency domain:
\begin{equation}
\label{eq:ForwardSDE_Frequency}
    d\hat{\mathbf{a}}_t = \hat{\mathbf{f}}(\hat{\mathbf{a}}_t,t) dt + g(t) d\tilde{\mathbf{v}}_t,
\end{equation}
where $\hat{\mathbf{f}}(\hat{\mathbf{a}},t) = U \mathbf{f}(Y\hat{\mathbf{a}},t)$, and $\tilde{\mathbf{v}} = U \mathbf{w}$ is a spherical mirrored Brownian motion on $\mathbb{C}^{L^2}$.

The associated reverse-time diffusion process is
\begin{equation}
\label{eq:ReverseSDE_Frequency}
    d\hat{\mathbf{a}}_t = \hat{\mathbf{b}}(\hat{\mathbf{a}}_t,t) dt + g(t) d\tilde{\mathbf{v}}_t,
\end{equation}
where
\[
    \hat{\mathbf{b}}(\hat{\mathbf{a}}_t,t) = \hat{\mathbf{f}}(\hat{\mathbf{a}}_t,t) - g^2(t)  \Sigma \, \hat{\mathbf{s}}(\hat{\mathbf{a}}_t,t),
\]
$\Sigma \in \mathbb{R}^{L^2 \times L^2}$ is the covariance matrix of the spherical
mirrored Brownian motion expressed in the real coordinate system induced by $\varphi$, $dt$ is a negative infinitesimal time step, and $\tilde{\mathbf{v}}_t$ is a spherical mirrored Brownian motion on $\mathbb{C}^{L^2}$ with time running backward from $T$ to $0$.
\end{theorem}

Theorem~\ref{thm:DiffusionFrequency} provides a practical recipe for implementing
diffusion models in the frequency domain for spherical signals.
The proof is given in Section~\ref{sec:proofs}.
It shows that the score-based SDE formalism introduced in
\cite{Song_NEURIPS2020,song20} extends naturally to this setting, with one important
modification: the standard Brownian motion driving the diffusion is replaced by a
\emph{spherical mirrored Brownian motion}.

\begin{remark}[Covariance structure and effective noise normalization]
\label{rem:SigmaLambda}

Lemma~\ref{lemma:SphericalDFT_BM} shows that the noise term
\[
\tilde{\mathbf{v}}(t) = U \mathbf{w}(t)
\]
appearing in Theorem~\ref{thm:DiffusionFrequency} is not a standard Brownian motion on
$\mathbb{C}^{L^2}$, but a centered Gaussian process with nontrivial covariance.
More precisely, 
the covariance operator of $\varphi(\tilde{\mathbf{v}}(t))$ is given by 

\[
\mathbb{E}\!\left[\varphi(\tilde{\mathbf{v}}(t))\varphi(\tilde{\mathbf{v}}(t))^{\mathsf T}\right]
= t\,\Sigma,
\qquad \text{with }
\Sigma\in \mathbb{R}^{L^2 \times L^2}.
\]

In particular, due to discretization effects induced by the spherical quadrature,
the coefficients $C_{(\ell,m),(\ell',m)}$ entering $\Sigma$ need not vanish for
$\ell \neq \ell'$. Consequently, the covariance matrix $\Sigma$ is block-structured
with respect to the index $m$, but is in general non-diagonal in $\ell$, as can be seen in Fig. \ref{fig:cov-phi-Uw}.

The covariance matrix $\Sigma$ is defined with respect to the real coordinate system
induced by the chart $\varphi$, which extracts the independent real degrees of freedom
of the spherical Fourier coefficients. With respect to these coordinates, the entries
of $\Sigma$ take the form
\begin{equation}
\label{eq: Sigma cov matrix}
   \Sigma_{(\ell,m),(\ell',m')}
=
\begin{cases}
2\,C_{(\ell,0),(\ell',0)}, & m = m' = 0,\\[4pt]
C_{(\ell,m),(\ell',m)}, & m = m' > 0,\\[4pt]
0, & \text{otherwise}.
\end{cases} 
\end{equation}

Here, for $m>0$, the real and imaginary parts of $a_{\ell,m}$ correspond to distinct
coordinates that share the same covariance coefficient $C_{(\ell,m),(\ell',m)}$.

As a consequence, $\Sigma$ is a real, symmetric, and positive semidefinite matrix
encoding the full covariance structure induced by the spherical harmonic basis,
the quadrature scheme, and the conjugate symmetry constraints.

To express the diffusion SDE in the standard score-based form, it is convenient to
introduce a matrix $\Lambda \in \mathbb{R}^{L^2 \times L^2}$ such that
\begin{equation}
\label{eq:SigmaFactorization}
\Sigma = \Lambda \Lambda^{\mathsf T}.
\end{equation}
The factorization~\eqref{eq:SigmaFactorization} is not unique; any fixed choice
(for example, a Cholesky factor) is admissible and is kept constant throughout the diffusion
process.

With this notation, the noise term admits the representation
\[
\varphi[\tilde{\mathbf{v}}(t)] \;\overset{d}{=}\; \Lambda \tilde{\mathbf{w}}(t),
\]
where $\tilde{\mathbf{w}}$ is a standard Brownian motion on $\mathbb{R}^{L^2}$.
Consequently, the forward diffusion SDE in the frequency domain may be equivalently
written as
\[
d\hat{\mathbf{a}}_t
=
\hat{\mathbf{f}}(\hat{\mathbf{a}}_t,t)\,dt
+
g(t)\,\Lambda\,d\tilde{\mathbf{w}}_t,
\]
and the reverse-time drift involves the term
\[
g^2(t)\,\Sigma\,\hat{\mathbf{s}}(\hat{\mathbf{a}}_t,t),
\]
where the score $\hat{\mathbf{s}}$ is taken with respect to the law of
$\hat{\mathbf{a}}$ under this non-isotropic noise.

This formulation makes explicit that the fundamental difference from the Euclidean
score-based diffusion framework of~\cite{song20} lies in the covariance structure
of the driving noise, which here reflects the spectral symmetries and redundancy
inherent to real-valued spherical signals.
\end{remark}

\begin{remark}[Consistency with reverse-time theory]
The forward spatial-domain SDE~\eqref{eq:forwardSDE} is a linear diffusion with
additive, state-independent noise. Since the spherical DFT operator $U$ is linear
and deterministic, applying $U$ to the forward SDE yields a forward diffusion
process in the frequency domain whose driving noise is the Gaussian process
$\tilde{\mathbf{v}}(t) = U\mathbf{w}(t)$.

Although $\tilde{\mathbf{v}}(t)$ takes values in a complex vector space and its
natural covariance is given by $\mathbb{E}[\tilde{\mathbf{v}}(t)\tilde{\mathbf{v}}(t)^{H}]$,
the diffusion is most conveniently described in the real coordinate system
induced by the chart $\varphi$, which extracts the independent real degrees of
freedom of the spherical Fourier coefficients.

With respect to these real coordinates, the noise process $\varphi(\tilde{\mathbf{v}}(t))$
has independent increments and deterministic covariance
\[
\mathbb{E}\!\left[\varphi(\tilde{\mathbf{v}}(t))\,\varphi(\tilde{\mathbf{v}}(t))^{\mathsf T}\right]
= t\,\Sigma,
\]
where the matrix $\Sigma$ is given by~\eqref{eq: Sigma cov matrix}.
The resulting frequency-domain diffusion therefore falls within the class of
linear diffusions with possibly degenerate, non-isotropic covariance considered
in the reverse-time theory of~\cite{anderson82}.

Applying Anderson’s formula in the real coordinate system defined by $\varphi$
yields the reverse-time drift correction term
$g^2(t)\,\Sigma\,\hat{\mathbf{s}}(\hat{\mathbf{a}},t)$.
\end{remark}

\subsection{Relationship between Score Matching Losses}
\label{subsec:losses}
The final step is to define an appropriate loss in both the spatial and frequency domains and examine their relationship.  

In the spatial domain, Proposition~\ref{prop:U_isom} justifies the following score matching objective:
\begin{align}
\label{eq: score_matching loss spatial domain}
    &\theta^* = \underset{\theta \in \Theta}{\mathrm{argmin}} \, \mathbb{E}_{t, \mathbf{x}(0), \mathbf{x}(t)} 
    \left[ \mathcal{L}_{SM}\left(s_{\theta}, s_{t|0}, \mathbf{x}, t \right) \right],\\
    &\mathcal{L}_{SM}\left(s_{\theta}, s_{t|0}, \mathbf{x}, t \right) := \left\| s_{\theta}(\mathbf{x},t)-s_{t|0}(\mathbf{x},t)\right\|_Q^2,
\end{align}
where the $Q$-norm $\|\cdot\|_Q$ accounts for the spherical topology.

In the frequency domain, the reverse diffusion process from Equation~\eqref{eq:ReverseSDE_Frequency} allows sampling of spherical signals via $\hat{\mathbf{b}}(\hat{\mathbf{a}},t)$, which involves the unknown score $\hat{\mathbf{s}}$. Analogously, we define an approximating score $\hat{\mathbf{s}}_{\hat{\theta}}$ and optimize the corresponding frequency-domain objective:
\begin{equation}
\label{eq: score_matching loss frequency spherical case}
\begin{split}
 &   \hat{\theta}^* = \underset{\hat{\theta} \in \Theta}{\mathrm{argmin}}  \mathbb{E}_{t, \hat{\mathbf{a}}(0), \hat{\mathbf{a}}(t)} 
    \left[ \mathcal{L}_{SM}\left(\hat{s}_{\hat{\theta}}, \Sigma  \hat{s}_{t|0}, \hat{\mathbf{a}}, t \right) \right],\\
&    \mathcal{L}_{SM}\left(\hat{s}_{\hat{\theta}}, \Sigma  \hat{s}_{t|0}, \hat{\mathbf{a}}, t \right) := \left\|\hat{s}_{\hat{\theta}}(\hat{\mathbf{a}},t)-\Sigma  \hat{s}_{t|0}(\hat{\mathbf{a}},t)\right\|_2^2,\end{split}
\end{equation}
with $t \sim \mathcal{U}(0,T)$, $\hat{\mathbf{a}}(0) \sim \hat{\mathbf{p}}_0$, $\hat{\mathbf{a}}(t) \sim \hat{\mathbf{p}}_{t|0}(\cdot \vert \hat{\mathbf{a}}(0))$, and $\Sigma$ the matrix from Theorem~\ref{thm:DiffusionFrequency}.  
In practice, this involves computing frequency representations of spherical maps, sampling from $\hat{p}_{t|0}$ via Equation~\eqref{eq:ForwardSDE_Frequency}, and then mapping the resulting complex signals back to the spatial domain using $\mathcal{F}^{-1}$.

A natural question arises: why does minimizing the frequency-domain loss in Equation~\eqref{eq: score_matching loss frequency spherical case} imply that $\hat{p}_0 \approx \hat{p}_{\text{data}}$?  
The key idea is to associate to $\hat{\mathbf{s}}_{\hat{\theta}}$ an auxiliary spatial-domain score
\[
s'_{\hat{\theta}}(\mathbf{x},t) := Y \hat{s}_{\hat{\theta}}(U\mathbf{x},t),
\]
where $Y$ and $U$ are the isometries from Proposition~\ref{prop:U_isom}. Unlike the Euclidean time-series setting of~\cite{crabbe24}, the spherical geometry and the induced noise covariance prevent an exact equivalence between spatial and frequency domain score matching, leading instead to an inequality that highlights a geometry-induced inductive bias.

\begin{theorem}[Score matching losses in spatial and frequency domains]
\label{thm:main}
Let $\hat{s}_{\hat{\theta}} : \mathbb{C}^{L^2} \times [0,T] \mapsto \mathbb{C}^{L^2}$ satisfy the spherical mirror symmetry 
$[\hat{s}_{\hat{\theta}}]_{(\ell,m)} = (-1)^m[\hat{s}_{\hat{\theta}}^*]_{(\ell,-m)}$ for all $0\leq \ell < L$ and $|m|< L$.  
Define the auxiliary spatial-domain score $s'_{\hat{\theta}} : \mathbb{R}^{2L(2L-1)} \times [0,T] \mapsto \mathbb{R}^{2L(2L-1)}$ as above. Then
\begin{equation}\label{eq:score_matching bound}
\begin{split}
\mathcal{L}_{SM}\left(\hat{s}_{\hat{\theta}}, \Sigma\hat{s}_{t|0}, \hat{\mathbf{a}}, t\right) 
  &  \leq 2
    \left\{ 
    \mathcal{L}_{SM}\left(s'_{\hat{\theta}}, s_{t|0}, \mathbf{x}, t\right)\right.\\& + \left.
    \left\| 
    U Z \Sigma\nabla_{\phi[\hat{\mathbf{a}}(t)]}
    \log p_{t|0}\left( \phi[\hat{\mathbf{a}}(t)]  \big\vert  \phi[\hat{\mathbf{a}}(0)] \right)
    \right\|^2
    \right\},
    \end{split}
\end{equation}
where 
\begin{equation*}
\begin{split}
& \hat{s}_{t|0}(\hat{\mathbf{a}}, t) = \nabla_{\hat{\mathbf{a}}(t)} \log \hat{p}_{t|0}\left( \hat{\mathbf{a}}(t)  \big \vert  \hat{\mathbf{a}}(0)\right),\\  
& s_{t|0}(\mathbf{x},t) = \nabla_{\mathbf{x}(t)} \log p_{t|0}\left( \mathbf{x}(t)\vert\mathbf{x}(0)\right),
\end{split}
\end{equation*}
and $\Sigma$ is as in Theorem~\ref{thm:DiffusionFrequency}.
\end{theorem}
This inequality formalizes a key difference between spherical and Euclidean spectral diffusion: the geometry-induced noise covariance prevents exact equivalence between spatial and frequency score matching, even in the band-limited setting.
Unlike~\cite{crabbe24}, where the losses are equivalent, 
Theorems~\ref{thm:DiffusionFrequency} and~\ref{thm:main} show that the frequency-domain loss is only upper-bounded by the spatial-domain loss.
Theorem~\ref{thm:main} therefore shows that diffusion in the spherical harmonic domain is not merely a change of coordinates of the spatial formulation. The geometry-induced covariance of the spectral noise leads to a genuinely different optimization problem and hence to a different inductive bias.
The proof is given in Section \ref{sec:proofs}.

\section{Numerical Illustrations}
\label{numerics}

The purpose of this section is to provide illustrative numerical experiments that complement the theoretical analysis developed in the previous sections. 
Rather than aiming at empirical benchmarking or large-scale performance evaluation, the experiments are designed to validate and visualize the stochastic and geometric effects induced by the spectral formulation of diffusion models on the sphere.

In particular, the numerical results focus on three aspects that play a central role in the theory: 
(i) the behavior of Brownian motion under the spherical discrete Fourier transform, 
(ii) the resulting non-isotropic and structured covariance of the driving noise in the frequency domain, and 
(iii) the practical implications of this geometry-induced noise structure for forward and reverse diffusion dynamics. 
All experiments are conducted in a finite-dimensional, band-limited setting, consistent with the framework of Section~\ref{sec:contribution}.

We first consider synthetic experiments that directly verify the theoretical characterization of the spherical mirrored Brownian motion, by comparing empirical estimates of the covariance matrix with the analytical expression derived in Section~\ref{sec:contribution}. 
These experiments serve as a sanity check of the discrete harmonic analysis and noise transformation underlying the spectral diffusion SDEs.

We then include a learning-based experiment on MNIST-Sphere, a widely used dataset obtained by mapping handwritten digit images onto the sphere. 
Although the theoretical results of this paper are independent of any specific data distribution, MNIST-Sphere provides a nontrivial and structured test case for spectral diffusion models: it exhibits strong low-frequency spectral bias together with localized, anisotropic spatial features, making it sensitive to the non-isotropic covariance structure induced in the frequency domain.
The goal of this experiment is not to claim improved generative performance, but to assess whether the spectral diffusion framework derived in this work remains numerically stable and produces samples of comparable quality when applied to realistic spherical data.

Throughout this section, numerical results should therefore be interpreted as qualitative and diagnostic, rather than as evidence of empirical superiority. 
They are intended to confirm the theoretical findings, illustrate the impact of geometry on diffusion dynamics, and provide intuition on how spectral diffusion on the sphere departs from its Euclidean counterpart through geometry-induced stochastic structure.

\subsection{Covariance Matrix of the spherical mirrored Brownian motion}
In this section, we report the plot of the covariance matrix $\Sigma$ of $\varphi[{\mathbf{\tilde{v}}}(t)]$ to highlight its particular structure, as defined in Eq. \eqref{eq: Sigma cov matrix}. 
\begin{figure}[htbp]
  \centering
  \includegraphics[width=\textwidth]{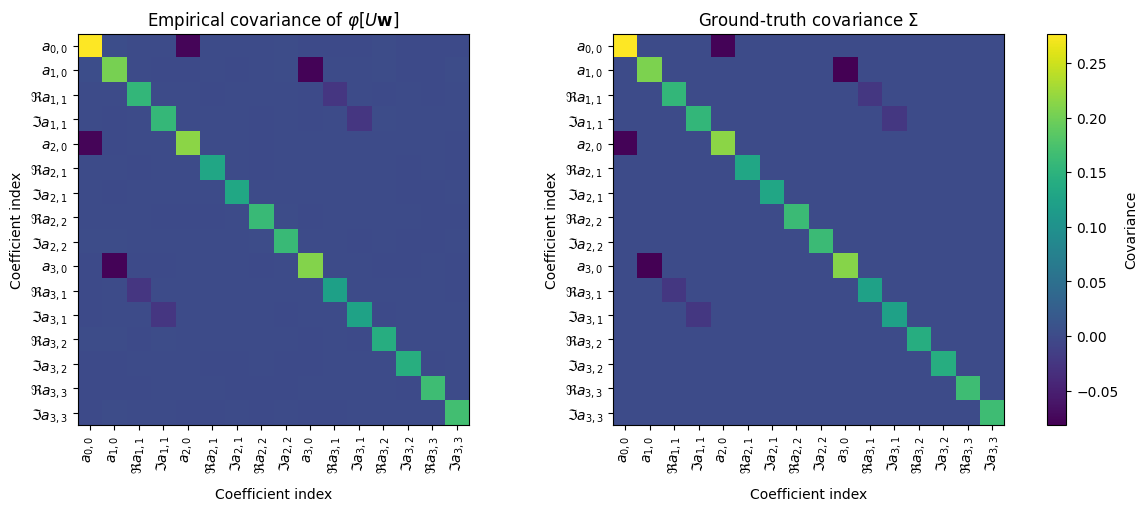}
  \caption{Comparison between the empirical covariance matrix of $\varphi[{\mathbf{\tilde{v}}}(t)]$, for $t=0$, estimated from $5\times10^{4}$ i.i.d samples (left), and the theoretical covariance matrix $\Sigma$ (right), with maximum bandlimit  $L=4$.}

  \label{fig:cov-phi-Uw}
\end{figure}

\subsection{Model.} For the experimental part, we parametrize the spatial score model $\mathbf{s}_\theta$ and the frequency score model 
$\hat{\mathbf{s}}_{\hat{\theta}}$ as transformer encoders with 10 attention and MLP layers, 
each with 12 heads and dimension $d_{\text{model}} = 72$. Both models have learnable 
positional encoding as well as diffusion time encoding through random Fourier features 
composed with a learnable dense layer. This results in a model with approximately 2M parameters. 
We use a VP-SDE with linear noise scheduling and $\beta_{\min} = 0.1$ and 
$\beta_{\max} = 10$. The score models are trained with the denoising 
score-matching loss, as defined in Section~\ref{subsec:losses}. The model is trained for $200$ epochs with batch size 32, AdamW optimizer, and cosine 
learning rate scheduling ($20$ warmup epochs, $lr_{\max} = 10^{-3}$).

\subsection{Spatial and Frequency Models.} The essential distinction between spatial-based and frequency-based diffusion models lies in the representation of their input spherical maps. Since all datasets are naturally 
defined in the spatial domain, they can be directly processed by the spatial-domain 
diffusion models~$\mathbf{s}_\theta$. For the frequency-domain diffusion models~$\hat{\mathbf{s}}_{\theta}$, 
each signal is first transformed into the frequency domain using a Spherical Discrete 
Fourier Transform (Spherical DFT).

In the spatial domain, the forward and reverse diffusion processes follow the SDEs 
given in Equations~\eqref{eq:forwardSDE} and~\eqref{eq:revdif}. In the frequency domain, the corresponding forward and reverse processes satisfy the modified SDEs in Equations~\eqref{eq:ForwardSDE_Frequency} and~\eqref{eq:ReverseSDE_Frequency}. 
The denoised samples~$\hat{\mathbf{x}}{(0)}$ obtained in the frequency domain can be 
converted back into the spatial domain by applying the inverse Spherical DFT: $\hat{x}{(0)} \;\mapsto\; Y\hat{x}{(0)}.$

In what follows, we denote by 
$S_{\text{spat}} \subset \mathbb{R}^{2L(2L-1)}$ and 
$S_{\text{freq}} \subset \mathbb{R}^{L^2}$ 
the spatial-domain and frequency-domain representations of the samples generated by the respective models. For each model, we generate $|S_{\text{time}}| = |S_{\text{freq}}| = 1{,}000$ samples by running the diffusion process for $T = 1{,}000$ timesteps.

\subsection{Learning-Based Illustration on MNIST-Sphere}

To complement the synthetic experiments and assess the practical implications of the proposed framework on structured spherical data, we consider a learning-based illustration on MNIST projected onto the unite sphere, which we'll call MNIST-Sphere from here on out.
MNIST-Sphere is obtained by mapping handwritten digit images from the MNIST dataset onto the sphere via a fixed spherical projection, following the construction
introduced in~\cite{cgc18}, and resulting in real-valued signals defined on $\mathbb{S}^2$. This construction has since become a standard benchmark for spherical generative
models.
While extrinsic in nature, this dataset is widely used as a controlled testbed for spherical generative models and provides a nontrivial combination of global structure and localized spatial features.

Importantly, MNIST-Sphere exhibits a strongly non-uniform spectral energy distribution, with most of its variance concentrated at low spherical harmonic degrees and pronounced anisotropy in the spatial domain.
As a result, it constitutes a challenging test case for diffusion processes driven by non-isotropic and geometry-induced noise in the frequency domain.
The purpose of this experiment is not to demonstrate improved generative performance, but to verify that the spectral diffusion SDEs derived in Section~\ref{sec:contribution} remain numerically stable and compatible with learning and sampling on realistic spherical data.

We train two diffusion models with identical architectures and optimization settings:
a spatial-domain model operating directly on sampled spherical maps, and a frequency-domain model acting on spherical harmonic coefficients subject to the mirror symmetry constraints described in Section~\ref{subsec:losses}.
Both models are implemented as transformer-based score networks with shared architectural hyperparameters and are trained using denoising score matching under a variance-preserving SDE.
All hyperparameters are kept fixed across representations to isolate the effect of the diffusion domain.

In the spatial domain, forward and reverse dynamics follow the standard SDEs in Equations~\eqref{eq:forwardSDE} and~\eqref{eq:revdif}.
In the frequency domain, diffusion is governed by the modified SDEs in Equations~\eqref{eq:ForwardSDE_Frequency} and~\eqref{eq:ReverseSDE_Frequency}, where the driving noise has nontrivial covariance $\Sigma$ induced by the spherical discrete Fourier transform.
Samples generated in the frequency domain are mapped back to the spatial domain via the inverse spherical DFT for evaluation.

To compare the distributions induced by the two models, we adopt the \emph{sliced Wasserstein distance} introduced in~\cite{bonneel2015}, which provides a computationally tractable proxy for distributional similarity in high-dimensional spaces.
With a slight abuse of notation, we denote the sliced Wasserstein distance between two empirical sample sets $S_1$ and $S_2$ as $SW(S_1,S_2)$.
For two distributions $\mu_1$ and $\mu_2$, it is defined as
\begin{equation}
    SW_p(\mu_1, \mu_2)
    :=
    \int_{\mathbb{S}^{d-1}}
    W_p\!\left(P_u \# \mu_1,\; P_u \# \mu_2\right)\,du,
    \label{eq:sw_def}
\end{equation}
where $\mathbb{S}^{d-1}$ denotes the unit sphere in dimension $d$, $P_u(x) = u \cdot x$ is the projection of $x$ onto $u$, $P_u \# \mu$ is the push-forward of $\mu$ by $P_u$, and $W_p$ is the Wasserstein distance of order $p$.
In practice, we approximate~\eqref{eq:sw_def} via Monte Carlo integration using $n=1{,}000$ random projections and $p=2$:
\begin{equation}
    \hat{SW}_p(\mu_1, \mu_2)
    :=
    \frac{1}{n}
    \sum_{i=1}^n
    W_p\!\left(P_{u_i} \# \mu_1,\; P_{u_i} \# \mu_2\right).
    \label{eq:sw_est}
\end{equation}

Table~\ref{tab:sliced_wasserstein_metrics} reports sliced Wasserstein distances as a diagnostic measure of distributional consistency between generated samples and the MNIST-Sphere data distribution, rather than as a competitive performance metric.
We stress that sliced Wasserstein distances computed in the spatial and frequency
domains live in spaces of different dimensionality; consequently, their absolute
values are not directly comparable across metric domains.

\begin{table}[ht]
\centering
\caption{Sliced Wasserstein distances (mean $\pm$ 2 standard errors) between generated samples and the MNIST-Sphere data distribution, computed in both the spatial and frequency domains.
Comparable values across diffusion domains indicate that the geometry-induced, non-isotropic noise structure in the frequency-domain SDE does not prevent stable training or sampling, rather than signaling improved generative performance.}
\label{tab:sliced_wasserstein_metrics}
\begin{tabular}{l l cc}
\hline
\textit{Dataset} & \textit{Metric Domain} & \multicolumn{2}{c}{\textit{Diffusion Domain}} \\
 &  & \textit{Frequency} & \textit{Spatial} \\
\hline
\multirow{2}{*}{\textbf{MNIST-Sphere}} 
& Frequency & 0.0069 $\pm$ 0.0002 & 0.0074 $\pm$ 0.0002 \\
& Spatial   & 0.036  $\pm$ 0.001  & 0.035  $\pm$ 0.001 \\
\hline
\end{tabular}

\vspace{0.3em}
\footnotesize
\textit{Note:} Absolute values are not directly comparable across metric domains due to the different dimensionalities of the spatial and frequency representations.
\end{table}

Overall, the MNIST-Sphere experiment serves as a robustness check demonstrating that,
despite the non-isotropic noise structure and the lack of equivalence between spatial and frequency score matching established theoretically (Theorem~\ref{thm:main}),
spectral diffusion on the sphere remains numerically viable and induces sample distributions of comparable quality to its spatial-domain counterpart.

\section{Proofs}\label{sec:proofs}
This section provides detailed proofs of the main results stated in the paper, highlighting the underlying assumptions and key derivations.

\begin{proof}[Proof of Lemma \ref{lemma:DFT_mat_rep}]
Recall that $Q = Q_\theta \otimes I_{N_\phi}$, so that in the product $Y^H Q$ each set of $N_\phi$ columns corresponding to colatitude $\theta_j$ is multiplied by the weight $q_j$. Then the $(\ell,m)$-th entry of the product $(Y^H Q)\mathbf{x}$ reads
\[
\left[(Y^H Q)\mathbf{x}\right]_{(\ell,m)}
=
\sum_{j=0}^{N_\theta-1} \sum_{k=0}^{N_\phi-1} q_j  Y^*_{\ell,m}(\theta_j,\phi_k)  f(\theta_j,\phi_k),
\]
which coincides exactly with the analysis formula for $\hat{a}_{\ell,m}$. Hence, $\hat{\mathbf{a}} = U\mathbf{x}$, as claimed by Equation~\ref{eqn:lemma1}.
\end{proof}

\begin{proof}[Proof of Lemma \ref{lemma:pseudoinverseU}]
We first prove the right pseudoinverse property. By definition, $U = Y^H Q$, and the quadrature is exact for $L$-band-limited functions. Therefore, we have
\[
UY = Y^H Q Y = I_{L^2}.
\]

For the left pseudoinverse, let $P := YU = Y Y^H Q$. 

First, assume that $\mathbf{x}$ is $L$-band-limited, meaning $\mathbf{x} = Y \hat{\mathbf{a}}$ for some $\hat{\mathbf{a}} \in \mathbb{C}^{L^2}$. Then
\[
P \mathbf{x} = Y Y^H Q (Y \hat{\mathbf{a}}) = Y (Y^H Q Y) \hat{\mathbf{a}} = Y I_{L^2} \hat{\mathbf{a}} = Y \hat{\mathbf{a}} = \mathbf{x},
\]
so $1 \Rightarrow 2$.

Conversely, assume that $P \mathbf{x} = \mathbf{x}$. By definition, each column of $P$ is a linear combination of the columns of $Y$, so $\operatorname{range}(P) \subseteq \operatorname{range}(Y)$. On the other hand, for any $\mathbf{v} \in \operatorname{range}(Y)$, we can write $\mathbf{v} = Y \mathbf{a}$ for some $\mathbf{a} \in \mathbb{C}^{L^2}$, and then
\[
P \mathbf{v} = Y Y^H Q (Y \mathbf{a}) = Y (Y^H Q Y)  \mathbf{a} = Y I_{L^2} \mathbf{a} = Y \mathbf{a} = \mathbf{v},
\]
showing that $\mathbf{v} \in \operatorname{range}(P)$ and thus $\operatorname{range}(Y) \subseteq \operatorname{range}(P)$. We conclude that $\operatorname{range}(P) = \operatorname{range}(Y)$. Consequently, if $P \mathbf{x} = \mathbf{x}$, then $\mathbf{x} \in \operatorname{range}(Y)$, so there exists $\hat{\mathbf{a}} \in \mathbb{C}^{L^2}$ with $\mathbf{x} = Y \hat{\mathbf{a}}$, which shows that $\mathbf{x}$ is $L$-band-limited, which proves the implication $2 \Rightarrow 1$.
\end{proof}

\begin{proof}[Proof of Proposition \ref{prop:U_isom}]
To show that $U$ is an isometry, we need to prove that it preserves the inner product between the two spaces.  

Let $\mathbf{x}_1, \mathbf{x}_2 \in C^{d_X}_{\mathrm{band\text{-}limited}}$. By definition, there exist $\hat{\mathbf{a}}_1, \hat{\mathbf{a}}_2 \in \mathbb{C}^{\hat{d}_X}$ such that
\[
\mathbf{x}_1 = Y \hat{\mathbf{a}}_1, \quad \mathbf{x}_2 = Y \hat{\mathbf{a}}_2.
\]

Then
\begin{align*}
\langle U \mathbf{x}_1, U \mathbf{x}_2 \rangle 
&= (U \mathbf{x}_1)^H (U \mathbf{x}_2) \\
&= \mathbf{x}_1^H U^H U \mathbf{x}_2 \\
&= \mathbf{x}_1^H U^H Y^H Q Y U \mathbf{x}_2 \quad \text{(using the right pseudoinverse: $UY = I_{\hat{d}_X}$)} \\
&= \hat{\mathbf{a}}_1^H Y^H Q Y \hat{\mathbf{a}}_2 \\
&= \mathbf{x}_1^H Q \mathbf{x}_2 \\
&= \langle \mathbf{x}_1, \mathbf{x}_2 \rangle_Q.
\end{align*}

Hence, $U$ preserves the inner product and is therefore an isometry.
\end{proof}

\begin{proof}[Proof of Lemma~\ref{lemma:SphericalDFT_BM}]
We prove each property in turn.

\textit{(1) Symmetry.} This property follows directly from the symmetry of the spherical DFT, which ensures that for all $(\ell,m)$,
\[
\mathbf{v}_{(\ell,m)} = (-1)^m \mathbf{v}_{(\ell,-m)}^*.
\]

\textit{(2) Real Brownian motions.} Consider the components $\mathbf{v}_{(\ell,0)}$. By symmetry, these are real-valued. To verify that they are standard Brownian motions up to scaling, we first decompose the DFT operator as $U = U_{re} + i  U_{im}$ and compute the covariance structure. Define
\[
U_{re}U_{re}^{\sf T} \quad \text{and} \quad U_{im}U_{im}^{\sf T}.
\]  
A direct computation gives
\[
[U_{re}U_{re}^{\sf T}]_{(\ell,0),(\ell',0)} = 2 C_{(\ell,0),(\ell',0)}, \quad
[U_{im}U_{im}^{\sf T}]_{(\ell,0),(\ell',0)} = 0, \quad
[U_{re}U_{im}^{\sf T}]_{(\ell,0),(\ell',0)} = 0,
\]
where
\[
C_{(\ell,m),(\ell',m')} := \frac{2L-1}{2} N_{\ell,m} N_{\ell',m'} \sum_{j=0}^{2L-1} q_j^2 P_{\ell,m}(\cos \theta_j) P_{\ell',m'}(\cos \theta_j).
\]

\noindent Now, for $s,t\ge 0$ the increment $\mathbf{v}_{(\ell,0)}(t+s) - \mathbf{v}_{(\ell,0)}(s) = U\left( \mathbf{w}(t+s)-\mathbf{w}(s)\right)$ is Gaussian with mean zero and variance
\[
\operatorname{\mathbb{V}ar}\left(\mathbf{v}_{(\ell,0)}(t+s)-\mathbf{v}_{(\ell,0)}(s)\right) = t [UU^H]_{(\ell,0),(\ell,0)} = t2C_{(\ell,0),(\ell,0)}.
\]  
The continuity and independence of increments follow from the linearity of $U$ and the corresponding properties of $\mathbf{w}$. Hence
\[
\mathbf{v}_{(\ell,0)}/\sqrt{2C_{(\ell,0),(\ell,0)}}
\]
is a standard real Brownian motion.

\textit{(3) Complex Brownian motions.} Consider $0 < m < L$. Decompose $\mathbf{v}_{(\ell,m)}$ into real and imaginary parts: $\mathbf{v}_{(\ell,m)} = \Re(\mathbf{v}_{(\ell,m)}) + i\Im(\mathbf{v}_{(\ell,m)})$. Using the previously computed covariances, we have
\[
\operatorname{\mathbb{V}ar}\left( \Re(\mathbf{v}_{(\ell,m)})(t+s)-\Re(\mathbf{v}_{(\ell,m)})(s) \right) = tC_{(\ell,m),(\ell,m)},
\]
\[
\operatorname{\mathbb{V}ar}\left( \Im(\mathbf{v}_{(\ell,m)})(t+s)-\Im(\mathbf{v}_{(\ell,m)})(s)\right) = tC_{(\ell,m),(\ell,m)},
\]
\[
\operatorname{\mathbb{C}ov}\left( \Re(\mathbf{v}_{(\ell,m)}),\Im(\mathbf{v}_{(\ell,m)})\right) = 0.
\]  
Therefore, the real and imaginary parts are independent standard Brownian motions up to scaling, and we can write
\[
\mathbf{v}_{(\ell,m)} = \sqrt{C_{(\ell,m),(\ell,m)}} \left( \tilde{\mathbf{w}}^1_{(\ell,m)} + i  \tilde{\mathbf{w}}^2_{(\ell,m)} \right),
\]
with $\tilde{\mathbf{w}}^1_{(\ell,m)}$ and $\tilde{\mathbf{w}}^2_{(\ell,m)}$ independent standard Brownian motions.

\noindent This completes the proof.
\end{proof}

\begin{proof}[Proof of Theorem \ref{thm:DiffusionFrequency}]
We consider first the \emph{forward SDE}. 
Since $\hat{\mathbf{a}} = U\mathbf{x}$ is linear, its Jacobian is $U$ and its Hessian vanishes.
Applying the multivariate It\^o's lemma (see \cite[Equation (8.3)]{Kloeden1992}), we obtain
\begin{equation}
d\hat{\mathbf{a}} = U \mathbf{f}(\mathbf{x}, t)  dt + g(t) U  d\mathbf{w}.
\end{equation}
By Lemma \ref{lemma:SphericalDFT_BM}, $\tilde{\mathbf{v}} = U \mathbf{w}$ is a spherical mirrored Brownian motion on $\mathbb{C}^{L^2}$, giving the desired forward SDE.

We now consider the \emph{reverse-time SDE}. Following \cite{crabbe24}, we proceed in three steps: 
\begin{enumerate}
\item define a forward SDE for the truncated coordinates $\varphi[\hat{\mathbf{a}}]$, 
\item write the associated reverse-time SDE for $\varphi[\hat{\mathbf{a}}]$, and 
\item lift it back to the full process $\hat{\mathbf{a}}$.
\end{enumerate}

\textit{Step 1.} Using Equation \eqref{eq:ForwardSDE_Frequency} and Lemma \ref{lemma:SphericalDFT_BM}, the forward SDE for $\varphi[\hat{\mathbf{a}}]$ reads:
\begin{equation}
d\varphi[\hat{\mathbf{a}}] = \varphi\left[ U \mathbf{f}(Y \varphi^{-1}(\varphi[\hat{\mathbf{a}}]), t) \right]  dt + g(t) \Lambda d\tilde{\mathbf{w}},
\end{equation}
where
\begin{itemize}
    \item $\varphi^{-1}: \mathbb{R}^{L^2} \to \mathbb{C}^{L^2}$ satisfies     $\varphi^{-1}(\varphi[\hat{\mathbf{a}}]) = \hat{\mathbf{a}}, \quad \forall \hat{\mathbf{a}} \in \mathbb{C}^{L^2}_{\mathrm{constr}}.$
    \item $\Lambda \in \mathbb{R}^{L^2 \times L^2}$ is such that $\Sigma = \Lambda \Lambda^\top$, with 
    \[
    \Sigma_{(\ell,m),(\ell',m')} = 
    \begin{cases}
    2C(\ell,0;\ell',0), & m=m'=0,\\
    C(\ell,m;\ell',m), & m=m'>0,\\
    0, & \text{otherwise},
    \end{cases}
    \]
    and $\tilde{\mathbf{w}}$ is a standard Brownian motion in $\mathbb{R}^{L^2}$ such that $\Lambda \tilde{\mathbf{w}} = \varphi[U \mathbf{w}']$.
\end{itemize}

 \textit{Step 2.} The reverse-time SDE for $\varphi[\hat{\mathbf{a}}]$ \cite{anderson82} is
\begin{equation}
d\varphi[\hat{\mathbf{a}}] = \left\{ \varphi\left[ U \mathbf{f}(Y \varphi^{-1}(\varphi[\hat{\mathbf{a}}]), t) \right] 
- g(t)^2 \Lambda \Lambda^\top \nabla_{\varphi[\hat{\mathbf{a}}]} \log \hat{p}_t(\varphi[\hat{\mathbf{a}}]) \right\} dt + g(t) \Lambda 
d\hat{\mathbf{w}},
\end{equation}
with $\hat{\mathbf{w}}$ a standard Brownian motion in $\mathbb{R}^{L^2}$.

\textit{Step 3.} Applying $\varphi^{-1}$ and using It\^o's lemma gives the reverse SDE for $\hat{\mathbf{a}}$:
\begin{equation}
\begin{split}
d\hat{\mathbf{a}} &= \left\{ \varphi^{-1}\left( \varphi[U \mathbf{f}(Y \hat{\mathbf{a}}, t)] \right) 
- g(t)^2 \varphi^{-1}\left( \Sigma \nabla_{\varphi[\hat{\mathbf{a}}]} \log \hat{p}_t(\varphi[\hat{\mathbf{a}}]) \right) \right\} dt \\
&\quad + g(t) \varphi^{-1}(\Lambda)  d\hat{\mathbf{w}}.
\end{split}
\end{equation}
where we used that $\Lambda\Lambda^T = \Sigma$.
Since $\varphi^{-1}$ is linear, it commutes with deterministic linear operators such as $\Sigma$. Therefore, it holds that $$\Sigma \varphi^{-1}(y) = \varphi^{-1}(\Sigma y)$$ for every $y \in \mathbb{R}^{L^2}$. Then, the reverse-time SDE in the $\hat{\mathbf{a}}$-coordinates can be written as

\begin{equation}
d\hat{\mathbf{a}} = \left\{ U \mathbf{f}(Y \hat{\mathbf{a}}, t) - g(t)^2 \Sigma  \hat{\mathbf{s}}(\hat{\mathbf{a}}, t) \right\} dt + g(t)  d\tilde{\mathbf{v}},
\end{equation}
where $\tilde{\mathbf{v}} = \varphi^{-1}(\Lambda \hat{\mathbf{w}})$ is a spherical mirrored Brownian motion and $\hat{\mathbf{s}}$ is the score function defined in Equation \eqref{eq:score_freq}.
\end{proof}

\begin{proof}[Proof of Theorem \ref{thm:main}]
The proof proceeds in two steps. \textit{Step 1} establishes preliminary results expressing the score in terms of $\varphi[\hat{\mathbf{a}}]$. \textit{Step 2} uses these results to relate the frequency- and spatial-domain score matching losses.

\textit{Step 1: Score in terms of $\varphi[\hat{\mathbf{a}}]$.}  
Since $\mathbf{x} = Y \varphi^{-1}(\varphi[\hat{\mathbf{a}}])$ with $\hat{\mathbf{a}} = U \mathbf{x}$, using the change-of-variable formula yields
\[
p_{t|0}\left( \mathbf{x}(t) \big\vert \mathbf{x}(0)\right) = C \cdot \hat{p}_{t|0}\left( \varphi[\hat{\mathbf{a}}(t)]  \big \vert  \varphi[\hat{\mathbf{a}}(0)] \right),
\]
where $C$ is constant because $\mathbf{x} \mapsto \varphi[U\mathbf{x}]$ is linear.  
Define $\varphi[U\mathbf{x}] = V U_{\text{col}} \mathbf{x} = T \mathbf{x}$, where $V \in \mathbb{R}^{L^2 \times 2L^2}$, $U_{\text{col}} = \begin{pmatrix} U_{re} \\ U_{im} \end{pmatrix}$, and $T \in \mathbb{R}^{L^2 \times 2L(2L-1)}$.

\noindent Note preliminarily that 
\begin{itemize}
    \item For any $\mathbf{x}$, it holds that 
    \begin{equation}\label{eq:res0}
    \|U \mathbf{x}\|_2^2 \le \|\mathbf{x}\|_Q^2.
    \end{equation}
    Indeed, let $A := U Q^{-1/2}$ with $Q^{-1/2} Q^{-1/2} = Q^{-1}$. Then
    \[
    \|U\mathbf{x}\|_2^2 = \mathbf{x}^H U^H U \mathbf{x} = (Q^{1/2}\mathbf{x})^H A^H A (Q^{1/2}\mathbf{x}) \le \lambda_{\max}(A^H A) \|\mathbf{x}\|_Q^2.
    \]
    Using $U = Y^H Q$, we obtain $A A^H = I_{L^2}$, so $\lambda_{\max}(A^H A) = 1$.
    
    \item It holds that 
        \begin{equation}
        \label{eq: TT^t = Sigma}
            T T^{\sf T} = \Sigma
        \end{equation}
        This follows from $U_{\text{col}} U_{\text{col}}^{\sf T} = \mathrm{diag}(U_{re}^2, U_{im}^2)$, and $V$ extracts the relevant truncation indices.
    
    \item For any $\mathbf{y} \in \mathbb{R}^{L^2}$, it holds that
    \begin{equation}\label{eq:res2}
    T^{\sf T} \mathbf{y} = T^+ \Sigma \mathbf{y},
    \end{equation} 
    where $T^+ = T^T(TT^T)^{-1}$ is the Moore-Penrose pseudoinverse. Indeed, to obtain the result, simply multiply both sides of \eqref{eq: TT^t = Sigma} by $T^+$. Moreover, also $Y \varphi^{-1}$ acts as a right pseudoinverse of $T$:
    \[
    T Y \varphi^{-1}[\mathbf{y}] = \mathbf{y}, \quad \forall \mathbf{y}.
    \]
    So, in general, we can write
    \begin{equation}
    \label{eq:pseudoinverse_decomposition}
    Y \varphi^{-1} = T^+ + Z, \quad \text{with } TZ = 0.
    \end{equation}
    Here $Z$ collects the components of $Y\varphi^{-1}$ lying in $\ker(T)$.
\end{itemize}

Using these results, the chain rule gives
\begin{equation}
\label{eq:chain_rule_score}
\nabla_{\mathbf{x}(t)} \log p_{t|0}\left( \mathbf{x}(t)\big\vert \mathbf{x}(0) \right) = T^{\sf T} \nabla_{\varphi[\hat{\mathbf{a}}(t)]} \log \hat{p}_{t|0}\left( \varphi[\hat{\mathbf{a}}(t)] \big\vert \varphi[\hat{\mathbf{a}}(0)]\right).
\end{equation}

\textit{Step 2: Bounding the frequency-domain loss.}  
By Equation~\eqref{eq:score_freq} and using the triangle inequality $\|u+v\|_2^2 \le 2(\|u\|_2^2+\|v\|_2^2)$, it holds that
\begin{align*}
&\left\|\hat{s}_{\hat{\theta}}(\hat{\mathbf{a}},t) - \Sigma \hat{s}_{t|0}(\hat{\mathbf{a}},t)\right\|_2^2 \\
&\qquad= \left\|\hat{s}_{\hat{\theta}}(\hat{\mathbf{a}},t) - \Sigma \varphi^{-1} \nabla_{\varphi[\hat{\mathbf{a}}(t)]} \log \hat{p}_{t|0} \right\|_2^2 \\
&\qquad= \left\| \hat{s}_{\hat{\theta}}(\hat{\mathbf{a}},t) - U Y \varphi^{-1} \left(\Sigma\nabla_{\varphi[\hat{\mathbf{a}}(t)]} \log \hat{p}_{t|0} \right)\right\|_2^2 \\
&\qquad= \left\| \hat{s}_{\hat{\theta}}(\hat{\mathbf{a}},t) - U(T^+ + Z) \Sigma\nabla_{\varphi[\hat{\mathbf{a}}(t)]} \log \hat{p}_{t|0} \right\|_2^2 \\
&\qquad\le 2 \left( \left\| \hat{s}_{\hat{\theta}}(\hat{\mathbf{a}},t) - U T^+ \Sigma \nabla_{\varphi[\hat{\mathbf{a}}(t)]} \log \hat{p}_{t|0} \right\|_2^2 
+ \left\| U Z \Sigma\nabla_{\varphi[\hat{\mathbf{a}}(t)]} \log \hat{p}_{t|0} \right\|_2^2 \right).
\end{align*}
Note that the appearance of $\Sigma$ inside the loss (rather than as a metric) reflects the form of the reverse drift induced by the non-isotropic noise.

Finally, the first term can be upper-bounded using Equations \eqref{eq:res0} and \eqref{eq:res2}, together with the chain rule:
\begin{align*}
\left\| \hat{s}_{\hat{\theta}} - U T^+ \Sigma \nabla_{\varphi[\hat{\mathbf{a}}(t)]} \log \hat{p}_{t|0} \right\|_2^2 
&= \left\| U s'_{\hat{\theta}}(\mathbf{x},t) - U \nabla_{\mathbf{x}(t)} \log p_{t|0}\left(\mathbf{x}(t)\big\vert\mathbf{x}(0)\right) \right\|_2^2 \\
&\le \left\| s'_{\hat{\theta}}(\mathbf{x},t) - \nabla_{\mathbf{x}(t)} \log p_{t|0}\left( \mathbf{x}(t) \big\vert \mathbf{x}(0)\right) \right\|_Q^2.
\end{align*}
This concludes the proof.
\end{proof}

\section*{Funding}
CD was partially supported by the PRIN 2022 project GRAFIA (Geometry of Random Fields and its Applications), funded by the Italian Ministry of University and Research (MUR).

\bibliographystyle{plain}
\bibliography{bibliography}



\end{document}